\documentclass[10pt]{amsart}

\title{The Crepant Resolution Conjecture for Type $A$ Surface Singularities}

\author[Coates]{Tom Coates}
\address{Department of Mathematics\\
Imperial College London\\
180 Queen's Gate\\
London SW7 2AZ 
\\UK}
\email{t.coates@imperial.ac.uk}

\author[Corti]{Alessio Corti}
\address{Department of Mathematics\\
Imperial College London\\
180 Queen's Gate\\
London SW7 2AZ\\
UK}
\email{a.corti@imperial.ac.uk}

\author[Iritani]{Hiroshi Iritani}
\address{Faculty of Mathematics\\
Kyushu University\\
6-10-1, Hakozaki \\
Higashiku, Fukuoka, 812-8581 \\
Japan}
\email{iritani@math.kyushu-u.ac.jp}

\author[Tseng]{Hsian-Hua Tseng}
\address{Department of Mathematics\\
University of British Columbia\\
1984 Mathematics Road\\
Vancouver, B.C. V6T 1Z2\\
Canada}
\email{hhtseng@math.ubc.ca}

\usepackage{amsrefs}
\usepackage{amsfonts}
\usepackage{amssymb}
\usepackage{subfigure}

\allowdisplaybreaks[3]

\newcommand{\CC}{\mathbb{C}}
\newcommand{\ZZ}{\mathbb{Z}}
\newcommand{\RR}{\mathbb{R}}

\newcommand{\Cstar}{\CC^\times}
\newcommand{\cD}{\mathcal{D}}

\newcommand{\cO}{\mathcal{O}}
\newcommand{\ev}{\mathrm{ev}}

\renewcommand{\(}{\left(}
\renewcommand{\)}{\right)}

\newcommand{\cX}{\mathcal{X}}

\newcommand{\cIX}{\mathcal{IX}}

\newcommand{\HorbX}{H^\bullet_{\text{\rm orb}}(\cX;\CC)}
\newcommand{\HorbTX}{H^\bullet_{T,\text{\rm orb}}(\cX;\CC)}
\newcommand{\HorblocTX}{H(\cX)}

\newcommand{\HTY}{H^\bullet_T(Y;\CC)}
\newcommand{\HlocTY}{H(Y)}

\newcommand{\correlator}[1]{\left \langle #1 \right \rangle}
\newlength{\mybracketspacing}

\newcommand{\VHS}{V${\infty \over 2}$HS\ }

\newcommand{\CharV}{\mathfrak{V}}

\newcommand{\fr}[1]{\left\langle #1 \right\rangle}

\theoremstyle{plain}

\newtheorem{thm}{Theorem}
\newtheorem{pro}[thm]{Proposition}

\newtheorem*{lem*}{Lemma}

\newtheorem*{cor*}{Corollary}

\theoremstyle{definition}

\newtheorem*{exa*}{Example}

\newtheorem*{dfn*}{Definition}
\newtheorem*{rem*}{Remark}

\newcommand{\R}{\mathbb{R}}

\newcommand{\C}{\mathbb{C}}

\usepackage{amscd}

\theoremstyle{plain}

\newcommand{\cM}{\mathcal{M}}

\newcommand{\cU}{\mathcal{U}}

\newcommand{\hbeta}{\widehat{\beta}}

\def\parfrac#1#2{\frac{\partial #1}{\partial #2}}

\newcommand{\tw}{\text{\rm tw}}

\def\<{\left\langle}
\def\>{\right\rangle}

\begin{document}

\begin{abstract}
  Let $\cX$ be an orbifold with crepant resolution $Y$.  The Crepant
  Resolution Conjectures of Ruan and Bryan--Graber assert, roughly
  speaking, that the quantum cohomology of $\cX$ becomes isomorphic to
  the quantum cohomology of $Y$ after analytic continuation in certain
  parameters followed by the specialization of some of these
  parameters to roots of unity.  We prove these conjectures in the
  case where $\cX$ is a surface singularity of type $A$.  The key
  ingredient is mirror symmetry for toric orbifolds.
\end{abstract}

\maketitle

\section*{This Preprint is Obsolete}

Please note that this preprint has been superseded by {\tt
  arXiv:math/0702234v3} \cite{CCIT:lefschetz}.  The material here,
with various typos corrected, appears as Appendix A there.

\section{Introduction}

The small quantum cohomology of an orbifold $\cX$ is a family of
algebra structures on the Chen--Ruan orbifold cohomology $\HorbX$.
This family depends on so-called \emph{quantum parameters}, and encodes
certain genus-zero Gromov--Witten invariants of $\cX$.  A
long-standing conjecture of Ruan states that if $\cX$ is an orbifold
with coarse moduli space $X$ and $Y \to X$ is a crepant resolution
then the small quantum cohomology of $Y$ becomes isomorphic to the
small quantum cohomology of $\cX$ after analytic continuation in the
quantum parameters followed by specialization of some of the
parameters to roots of unity.  A refinement of this conjecture,
proposed recently by Bryan and Graber \cite{Bryan--Graber}, suggests
that if $\cX$ satisfies a Hard Lefschetz condition on orbifold
cohomology then the Frobenius manifold structures defined by the
quantum cohomology of $\cX$ and of $Y$ coincide after analytic
continuation and specialization of parameters (see also
\cite{CCIT:wallcrossings1} for a Hard Lefschetz condition).  This is a
stronger assertion: that the \emph{big} quantum cohomology of $Y$
coincides with that of $\cX$ after analytic continuation plus
specialization, via a linear isomorphism which preserves the
(orbifold) Poincar\'e pairing.  In this note we prove these
conjectures in the case where $\cX$ is the $A_{n-1}$ surface
singularity $\left[\CC^2/\mu_n\right]$ and $Y$ is its crepant
resolution.  In fact we prove a more precise statement,
Theorem~\ref{thm:maintheorem} below, which also identifies an
isomorphism and the roots of unity to which the quantum parameters of
$Y$ are specialized.  We learned this statement from Jim Bryan
\citelist{\cite{Bryan:personal} \cite{Bryan--Graber}*{Conjecture~3.1}}
and Fabio Perroni \citelist{\cite{Perroni:personal}
  \cite{Perroni}*{Conjecture~1.9}}. 

Our proof of Theorem~\ref{thm:maintheorem} is based on mirror symmetry
for toric orbifolds.  By mirror symmetry we mean the fact, first
observed by Candelas \emph{et~al.} \cite{COGP}, that one can compute
virtual numbers of rational curves in a manifold or orbifold $\cX$ ---
\emph{i.e.}  certain Gromov--Witten invariants of $\cX$ --- by solving
Picard--Fuchs equations.  Following Givental, we will formulate this
precisely as a relationship between a cohomology-valued generating
function for genus-zero Gromov--Witten invariants, called the
\emph{$J$-function of $\cX$}, and a cohomology-valued solution to the
Picard--Fuchs equations called the \emph{$I$-function of $\cX$}.  This
relationship is Proposition~\ref{pro:mirror} in
Section~\ref{sec:mirrorsymmetry}.  After describing the toric
structures of $\cX$ and $Y$ in Section~\ref{sec:toric} and fixing
notation for cohomology and quantum cohomology in
Section~\ref{sec:QC}, we explain in Section~\ref{sec:mirrorsymmetry}
how to extract the quantum products for $\cX$ and $Y$ from the
Picard--Fuchs equations.  Once we understand this,
Theorem~\ref{thm:maintheorem} follows easily: the proof is at the end
of Section~\ref{sec:mirrorsymmetry}.

A number of cases of Theorem~\ref{thm:maintheorem} were already known.
Ruan's Crepant Resolution Conjecture was established for surface
singularities of type $A_1$ and $A_2$ by Perroni \cite{Perroni}.
Theorem~\ref{thm:maintheorem} was proved in the $A_1$ case by
Bryan-Graber \cite{Bryan--Graber}, in the $A_2$ case by
Bryan--Graber--Pandharipande \cite{Bryan--Graber--Pandharipande}, and
in the $A_3$ case by Bryan--Jiang \cite{Bryan--Jiang}.  Davesh Maulik
has computed the genus-zero Gromov--Witten potential of the type $A$
surface singularity $\cX = \left[\CC^2/\mu_n\right]$ for all $n$ (as
well as certain higher-genus Gromov--Witten invariants of $\cX$) and
the reduced genus-zero Gromov--Witten potential of the crepant
resolution $Y$ \cite{Maulik}; Theorem~\ref{thm:maintheorem} should
follow from this.  The quantum cohomology of the crepant resolutions
of type $ADE$ surface singularities has been computed by
Bryan--Gholampour \cite{Bryan--Gholampour}.  Skarke \cite{Skarke} and
Hosono \cite{Hosono} have studied the $A_n$ case from a point of view
very similar to ours, as part of their investigations of homological
mirror symmetry.

\subsection*{Acknowledgements} We are grateful to Yongbin Ruan for
many productive and inspiring conversations, and to Jim Bryan, Etienne
Mann, and Fabio Perroni for useful discussions.  T.C. thanks Bong Lian
and Shing-Tung Yau for helpful conversations.  H.I. is grateful to
Akira Ishii for teaching him about autoequivalences of the derived
category.  The research on which this note is based took place at the
conference ``Quantum Cohomology of Stacks and String Theory'' at the
Institut Henri Poincar\'e.  T.C. is supported by the Royal Society and
by NSF grant DMS-0401275.  H.I. is supported by the Grant-in-Aid for
Scientific Research 18-15108 and the 21st Century COE program of
Kyushu University.  H.-H.T. thanks Institut Mittag-Leffler for
hospitality and support.

\section{$\cX$ and $Y$ as Toric Orbifolds}
\label{sec:toric}

$\cX$ is the toric orbifold corresponding to the fan\footnote{$\cX$ is
  also the toric Deligne--Mumford stack \cite{Borisov--Chen--Smith}
  corresponding to the stacky fan in Figure~\ref{fanforX}.} in
Figure~\ref{fanforX} and $Y$ is the toric manifold corresponding to
the fan in Figure~\ref{fanforY}.  Background material on toric
manifolds and orbifolds can be found in \cite[Chapter VII]{Audin}.
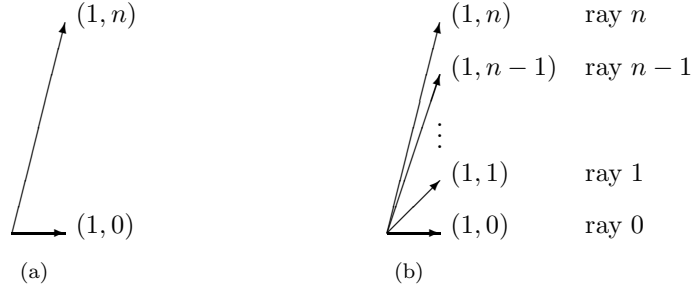
\begin{figure}
  \centering
  \subfigure[]
  {
    \label{fanforX}
    \begin{picture}(100,90)(-40,-5)
      \put(24,0){$\textstyle
        (1,0)$}
      \put(24,80){$\textstyle (1,n)$}
      \put(0,0){\vector(1,0){20}}
      \put(0,0){\vector(1,4){20}}
    \end{picture}
  }
  \hspace{1cm} 
  \subfigure[]
  {
    \label{fanforY}
    \begin{picture}(100,90)(-40,-5)
      \put(20,40){\makebox(0,0){$\vdots$}}
      \put(24,0){$\textstyle (1,0)$}
      \put(24,20){$\textstyle (1,1)$}
      \put(24,60){$\textstyle (1,n-1)$}
      \put(24,80){$\textstyle (1,n)$}
      \put(75,0){ray $0$}
      \put(75,20){ray $1$}
      \put(75,60){ray $n-1$}
      \put(75,80){ray $n$}
      \put(0,0){\vector(1,0){20}}
      \put(0,0){\vector(1,1){20}}
      \put(0,0){\vector(1,3){20}}
      \put(0,0){\vector(1,4){20}}
    \end{picture}
  }
  \caption{(a) The fan for $\cX$.   (b) The fan for $Y$.}
\end{figure}

There is an exact sequence
\[
\begin{CD}
  0 @>>> \ZZ^{n-1} @>{
  M^\mathrm{T}}>>  
  \ZZ^{n+1} @>{
    \begin{pmatrix} \textstyle
      1 & 1 & 1 & \cdots & 1 \\
      0 & 1 &  2 & \cdots & n
    \end{pmatrix}}>> \ZZ^2 @>>> 0,
\end{CD}
\]
and hence we can represent the Gale dual of the right-hand map by
\[
\begin{CD}
  \ZZ^{n+1} @>{M}>> \ZZ^{n-1},
\end{CD}
\]
where 
\[
M= 
 \begin{pmatrix}
   1 & -2 & 1 & 0 & 0 & \cdots & 0 \\
   0 & 1 & -2 & 1  & 0 & \cdots & 0 \\
   \vdots & & \ddots & & \ddots & & \vdots \\
   0 & \cdots & 0 & 1 & -2 & 1 & 0 \\
   0 & \cdots & 0 & 0 & 1 & -2 & 1
 \end{pmatrix}.
\]
Certain faces of the positive orthant $(\RR_{\geq 0})^{n+1} \subset
\R^{n+1}$ project via $M$ to codimension-$1$ subsets of $\RR^{n-1}$.
The image of the positive orthant is divided by these subsets into
chambers, which are the maximal cones of a fan in $\RR^{n-1}$ called
the \emph{secondary fan} of $Y$.  Chambers in the secondary fan
correspond to toric partial resolutions of $\cX$.  A chamber $K$
corresponds to a fan $\Sigma$ with rays some subset of the rays of the
fan for $Y$, as follows.  Number the rays of the fan for $Y$ as shown
in Figure~\ref{fanforY}.  For a subset $\sigma \subset
\{0,1,\ldots,n\}$, let us write $\bar{\sigma}$ for the complement
$\{0,1,\ldots,n\}\setminus \sigma$, $\RR^\sigma$ for the corresponding
co-ordinate subspace of $\RR^{n+1}$, and say that $\sigma$ covers $K$
iff $K \subset M(\RR^\sigma)$.  The fan $\Sigma$ corresponding to the
chamber $K$ is defined by
\[
  \sigma \in \Sigma  \iff  \text{$\bar{\sigma}$ covers $K$};
\]
the chamber $K$ corresponding to the fan $\Sigma$ is
\[
\bigcap_{\sigma \in \Sigma} M\(\RR^{\bar{\sigma}}\).
\]
We will concentrate on two chambers: $K_{\cX}$, with rays given by the
middle $n-1$ columns of $M$, and $K_Y$ with rays given by the standard
basis vectors for $\RR^{n-1}$.  $K_{\cX}$ corresponds to the toric
orbifold $\cX$ and $K_Y$ corresponds to the toric manifold $Y$.

Let $\cM_{\text{\rm sec}}$ be the toric orbifold corresponding to the
secondary fan of $Y$.  As $K_{\cX}$ and $K_Y$ are simplicial, they
give co-ordinate patches on $\cM_{\text{\rm sec}}$: the co-ordinates
$x_1,\ldots,x_{n-1}$ from $K_{\cX}$ and $y_1,\ldots,y_{n-1}$ from
$K_Y$ are related by
\begin{subequations}
  \label{eq:xtoy}
  \begin{equation}
    y_i =
    \begin{cases}
      x_1^{-2} x_2 & i=1 \\
      x_{i-1} x_i^{-2} x_{i+1} & 1 < i < n-1 \\
      x_{n-2} x_{n-1}^{-2} & i=n-1.
    \end{cases}
  \end{equation}
  More precisely, $x_1,\dots,x_{n-1}$ are multi-valued and the
  co-ordinate patch $\cM_{\rm sec}(K_\cX)$ corresponding to the cone
  $K_{\cX}$ is given by the uniformizing system:
  \[
  \cM_{\rm sec}(K_\cX) \cong \C^{n-1}/\mu_n, \quad
  (x_1,x_2,\dots,x_{n-1}) \sim (cx_1, c^2 x_2,\dots, c^{n-1}x_{n-1}) \text{
    for } c\in \mu_n.
  \]
  The \emph{$B$-model moduli space $\cM_B$} is the open subset
  $\C\times \cM_{\rm sec}(K_\cX)$ of $\C\times \cM_{\rm sec}$.  Denote
  by $x_0$ or $y_0$ the co-ordinate on the first factor $\C$ of
  $\C\times \cM_{\rm sec}$, so that
  \begin{equation}
    x_0=y_0.
  \end{equation}
\end{subequations}
We will refer to the point $(x_0,x_1,\ldots,x_{n-1}) = (0,0,\ldots,0)$
as the \emph{large-radius limit point for $\cX$} and the point
$(y_0,y_1,\ldots,y_{n-1}) = (0,0,\ldots,0)$ as the \emph{large-radius
  limit point for $Y$}.  The co-ordinates $x_i$ and $y_j$ are related
to each other by \eqref{eq:xtoy}, so that $y_0,y_1,\ldots,y_{n-1}$ are
co-ordinates on the patch $\CC \times \(\Cstar\)^{n-1} \subset \CC
\times \cM_{\rm sec}(K_\cX) = \cM_B$ where each of
$x_1,x_2,\ldots,x_{n-1}$ is non-zero.

\begin{rem*}
  In what follows the first factor of $\cM_B$, which has co-ordinates
  $x_0$ or $y_0$, will play a rather different role than the second
  factor.  The first factor will correspond under mirror symmetry to
  $H^0_{\text{\rm orb}}(\cX) \subset H^\bullet_{\text{\rm orb}}(\cX)$
  or $H^0(Y) \subset H^\bullet(Y)$, and the second factor will
  correspond to $H^2_{\text{\rm orb}}(\cX) \subset
  H^\bullet_{\text{\rm orb}}(\cX)$ or $H^2(Y) \subset H^\bullet(Y)$.
\end{rem*}

\begin{rem*}
  It would be more honest to define the $B$-model moduli space 
  as the product of $\CC$ with the open subset of $\cM_{\text{\rm
      sec}}$ on which the GKZ system associated to $Y$ is
  non-singular.  This set is slightly smaller than $\cM_B$, as it does
  not contain the discriminant locus of $W_{\cX}$ or $W_Y$ which
  appears below (in the proof of Proposition~\ref{pro:ac}).
\end{rem*}

The presentations of $\cX$ as a toric orbifold and $Y$ as a toric
variety allow us to write $\cX$ and $Y$ as quotients of open sets
$\cU_\cX, \cU_Y \subset \CC^{n+1}$ by $\(\Cstar\)^{n-1}$.  The action
of $T = (\Cstar)^2$ on $\CC^{n+1}$ given by
\begin{equation}
  \label{eq:action}
  (a_0,a_1,\ldots,a_n) \overset{(s,t)}{\longmapsto}
  (sa_0,a_1,a_2,\ldots,a_{n-1}, t a_n)
\end{equation}
descends to give $T$-actions on $\cX$, $X$, and $Y$, and the crepant
resolution $Y \to X$ is $T$-equivariant.  The $T$-fixed locus on $Y$
is the exceptional divisor.  The $T$-action on $\cX =
\left[\CC^2/\mu_n\right]$ coincides with that induced by the standard
action of $T$ on $\CC^2$, so the $T$-fixed locus on $\cX$ is the
$B\mu_n$ at the origin.  We write $H_T^\bullet(\{{\rm pt}\}) =
\CC[\lambda_1,\lambda_2]$ where $\lambda_i$ is Poincar\'e-dual to a
hyperplane in the $i$th factor of $\(\mathbf{C P}^\infty\)^ 2 \simeq
BT$.

\section{Cohomology and Quantum Cohomology}
\label{sec:QC}

We will assume familiarity with quantum orbifold cohomology, referring
the reader to \cite[Section~2]{CCLT} for a brief overview and the
original sources \cite{Chen--Ruan:GW, AGV:2} for a detailed
exposition.  We will assume also familiarity with the work of
Bryan--Graber \cite{Bryan--Graber}, and in particular with their
enhanced notion of the degree of a stable map to an orbifold (``degree in
the twisted sectors'').  Our notation will be compatible with that in
\cite{Bryan--Graber}.

The $T$-equivariant orbifold cohomology $\HorbTX$ is the
$T$-equivariant cohomology of the inertia stack $\cIX$.  $\cIX$ has
components $\cX_0$, $\cX_1$, \ldots,$\cX_{n-1}$, where
\begin{align*}
  \cX_k = \left[\(\CC^2\)^g / \mu_n \right] && \text{with $g = \exp\(2
    k \pi \sqrt{-1} / n\) \in \mu_n$.}
\end{align*}
We have
\begin{align*}
  &\cX_k = \left[ \CC^2 / \mu_n \right] &  \text{age = $0$} &&&  \text{if $k=0$,} \\
  & \cX_k = B \mu_n &  \text{age = $1$} &&&  \text{otherwise.}
\end{align*}
Let $\delta_i$ be the fundamental class of $\cX_i$, $0 \leq i < n$;
this gives a $\CC[\lambda_1,\lambda_2]$-basis for $\HorbTX$.  The
canonical involution $I$ on $\cIX$ fixes $\cX_0$ and exchanges $\cX_i$
with $\cX_{n-i}$, $1 \leq i < n$.  As $I$ is age-preserving, $\HorbX$
satisfies Hard Lefschetz
\citelist{\cite{Bryan--Graber}*{Definition~1.1} \cite{Fernandez}}.

The cone $K_Y$ is the K\"ahler cone for $Y$ and its rays determine a
basis $\gamma_1,\ldots,\gamma_{n-1}$ for $H^2(Y;\ZZ)$.  The dual basis
$\beta_1,\ldots,\beta_{n-1}$ for $H_2(Y;\ZZ)$ is \emph{positive} in
the sense of \cite[Section~1.2]{Bryan--Graber}.  If we define
$\gamma_0 = 1$ and choose lifts of $\gamma_1,\ldots, \gamma_{n-1}$ to
$T$-equivariant cohomology then $\gamma_i$, $0 \leq i < n$, is an
$\CC[\lambda_1,\lambda_2]$-basis for $\HTY$.  We choose a standard
equivariant lift of each $\gamma\in H^2(Y;\ZZ)$ in the following way.
There is a unique representation $\rho_\gamma$ of $\(\Cstar\)^{n-1}$
such that $\gamma$ is the first Chern class of the line bundle
\[
L_\gamma:=\cU_Y\times_{\rho_\gamma}\C \longrightarrow \cU_Y/(\C^\times)^{n-1} =Y.
\]
This line bundle $L_\gamma$ admits a $T$-action such that $T$ acts on
$\cU_Y$ via \eqref{eq:action} and acts trivially on the $\C$ factor, and
the lift $\gamma\in H^2_T(Y;\ZZ)$ is the $T$-equivariant first Chern
class of $L_\gamma$.  The columns of $M$, together with the action
\eqref{eq:action}, define elements $\omega_j \in H^2_T(Y;\CC)$, $0
\leq j \leq n$, where
\[
\omega_j =
\begin{cases}
  \lambda_1 + \gamma_1 & j=0 \\
  -2 \gamma_1 + \gamma_2 & j=1 \\
  \gamma_{j-1} - 2 \gamma_j + \gamma_{j+1} & 1<j<n-1 \\
  \gamma_{n-2} - 2 \gamma_{n-1} & j=n-1 \\
  \lambda_2 + \gamma_{n-1} & j=n.
\end{cases}
\]
The class $\omega_i$ is the $T$-equivariant Poincar\'{e} dual of the toric
divisor given in co-ordinates \eqref{eq:action} by $a_i=0$.  We have
\begin{align*}
  \HTY &= \CC[\lambda_1,\lambda_2,\gamma_1,\ldots,\gamma_{n-1}]/ \left
    \langle \omega_i \omega_j : i-j>1 \right \rangle.
\end{align*}

$\cX$ and $Y$ are non-compact but nonetheless, as discussed in
\cite{Bryan--Graber}, one can define (orbifold) Poincar\'e pairings on
the localized $T$-equivariant (orbifold) cohomology groups
\begin{align*}
  \HorblocTX := H^\bullet_{T,\text{\rm orb}}(\cX;\CC) \otimes
  \CC(\lambda_1,\lambda_2) 
  && \text{and} && 
  \HlocTY := H_T^\bullet(Y;\CC)\otimes \CC(\lambda_1,\lambda_2) 
\end{align*}
using the Bott residue formula.  These pairings take values in
$\CC(\lambda_1,\lambda_2)$, and are non-degenerate.  
Similarly, even though some moduli spaces of stable maps to $\cX$ or
$Y$ are non-compact, the $T$-fixed loci on these moduli spaces are
compact and so we can still define $\CC(\lambda_1,\lambda_2)$-valued
Gromov--Witten invariants of $\cX$ and $Y$ using the virtual
localization formula of Graber--Pandharipande
\cite{Graber--Pandharipande}.  For $\alpha_1,\ldots,\alpha_n \in
\HlocTY$, $\beta \in H_2(Y;\ZZ)$, and $i_1,\ldots,i_n \geq 0$, we set
\[
\correlator{\alpha_1 \psi_1^{i_1}, \ldots, \alpha_n
  \psi_n^{i_n}}^Y_\beta = 
\int_{\left[\overline{M}_{0,n}(Y,\beta)\right]^{\text{\rm vir}}} \prod_{j=1}^n \ev_j^\star
  \alpha_j \cdot \psi_j^{i_j}.
\]
Here $\psi_1,\ldots,\psi_n$ are the universal cotangent line classes
on the moduli space $\overline{M}_{0,n}(Y,\beta)$ of genus-zero
$n$-pointed stable maps to $Y$ of degree $\beta$; the integral is
defined by localization to the $T$-fixed substack, as in
\cite[Section~4]{Graber--Pandharipande} or \cite[Section~3.1]{Coates}.
Bryan and Graber define a set of \emph{effective classes in the
  orbifold Neron--Severi group} of $\cX$, and for each effective class
$\hbeta$ describe an associated moduli space
$\overline{M}_{0,n}(\cX,\hbeta)$ of genus-zero $n$-pointed stable maps
to $\cX$.  As the notation suggests, one can think of $\hbeta$ as
recording some sort of degree of a stable map to $\cX$: an effective
class consists of a non-negative integer $\hbeta(i)$ for each inertia
component $\cX_i$, $1 \leq i < n$, and a stable map in
$\overline{M}_{0,n}(\cX,\hbeta)$ carries $\hbeta(i)$ extra unordered
marked points which are constrained to map to $\cX_i$.  We write
\[
\correlator{\alpha_1 \psi_1^{i_1}, \ldots, \alpha_n
  \psi_n^{i_n}}^\cX_{\hbeta} = 
\int_{\left[\overline{M}_{0,n}(\cX,\hbeta)\right]^{\text{\rm vir}}} \prod_{j=1}^n \ev_j^\star
  \alpha_j \cdot \psi_j^{i_j},
\]
where $\alpha_1,\ldots,\alpha_n \in \HorblocTX$; $i_1,\ldots,i_n \geq
0$; $\psi_1,\ldots,\psi_n$ are the universal cotangent line classes on
$\overline{M}_{0,n}(\cX,\hbeta)$; and the integral is once again
defined by virtual localization.  These integrals are in fact certain
local Gromov--Witten invariants \cite{Chiang--Klemm--Yau--Zaslow}.

The genus-zero Gromov--Witten invariants defined here assemble to give
associative quantum products on $\HorblocTX$ and $\HlocTY$.  The small
quantum product for $\cX$ is the $\CC(\lambda_1,\lambda_2)$-algebra
defined by
\begin{equation}
  \label{eq:smallQCX}
  \delta_i \star \delta_j = \sum_{k=0}^{n-1} 
  \correlator{\delta_i,\delta_j,\delta_k}^\cX_{\widehat{0}} 
  \delta^k.
\end{equation}
Here $\widehat{0}$ is the orbifold Neron-Severi class with
$\widehat{0}(i) = 0$, $1 \leq i < n$, and $\{\delta^i\}$ denotes the basis
dual to $\{\delta_i\}$ under the orbifold Poincar\'e pairing.  The small
quantum product for $\cX$ coincides with the Chen--Ruan or orbifold
cup product \cite{Chen--Ruan:orbifold}.  The small quantum product for
$Y$ is the family of $\CC(\lambda_1,\lambda_2)$-algebras, depending on
parameters $q_1,\ldots,q_{n-1}$, defined by
\begin{equation}
  \label{eq:smallQCY}
  \gamma_i \star \gamma_j = \sum_{\beta} \sum_{k=0}^{n-1} 
  \correlator{\gamma_i,\gamma_j,\gamma_k}^Y_{\beta} 
  q_1^{d_1} \cdots q_{n-1}^{d_{n-1}} \gamma^k.
\end{equation}
The sum here is over classes $\beta = d_1 \beta_1 + \cdots + d_{n-1}
\beta_{n-1}$ with each $d_i \geq 0$, and $\{\gamma^i\}$ denotes the basis
dual to $\{\gamma_i\}$ under the Poincar\'e pairing.  It follows from the
discussion below that the right-hand side of \eqref{eq:smallQCY}
defines an analytic function of $q_1,\ldots,q_{n-1}$ in some
neighbourhood of the origin.  Ruan's conjecture asserts that the small
quantum cohomology algebra $\(\HlocTY,\star\)$ becomes isomorphic to
$\(\HorblocTX,\star\)$ after analytic continuation in the $q_i$
followed by setting the $q_i$ equal to certain roots of unity.

Bryan and Graber's refinement of Ruan's conjecture involves big
quantum cohomology.  The big quantum cohomology of $\cX$ is the family
of $\CC(\lambda_1,\lambda_2)$-algebras parametrized by $u \in
\HorblocTX$, $u = u_0 \delta_0 + u_1 \delta_1 + \cdots + u_{n-1}
\delta_{n-1}$, defined by
\begin{equation}
  \label{eq:QCX}
 \delta_i \underset{\rm big}{\star} \delta_j = \sum_{\hbeta} \sum_{k=0}^{n-1} 
  \correlator{\delta_i,\delta_j,\delta_k}^\cX_{\hbeta} 
  u_1^{\hbeta(1)} \cdots u_{n-1}^{\hbeta(n-1)} \delta^k.
\end{equation}
The sum here is over effective classes $\hbeta$.  The big quantum
cohomology of $Y$ is the family of $\CC(\lambda_1,\lambda_2)$-algebras
parametrized by $t \in \HlocTY$, $t = t_0 \gamma_0 + t_1 \gamma_1 +
\cdots + t_{n-1} \gamma_{n-1}$, defined by
\begin{equation}
  \label{eq:QCY}
  \gamma_i \underset{\rm big}{\star}  \gamma_j = \sum_{\beta} \sum_{k=0}^{n-1} 
  \correlator{\gamma_i,\gamma_j,\gamma_k}^Y_{\beta} 
  e^{d_1 t_1 + \cdots + d_{n-1} t_{n-1}} \gamma^k.
\end{equation}
The sum here is over classes $\beta = d_1 \beta_1 + \cdots + d_{n-1}
\beta_{n-1}$ with each $d_i \geq 0$.  Together with the (orbifold)
Poincar\'e pairings, the big quantum cohomology algebras define
Frobenius manifolds\footnote{These Frobenius manifolds are defined
  over the field $\CC(\lambda_1,\lambda_2)$.} based on $\HorblocTX$
and $\HlocTY$.  The Bryan--Graber version of the Crepant Resolution
Conjecture asserts that these Frobenius manifolds coincide after
analytic continuation in the $t_i$ and an appropriate
change-of-variables.  This is our main result.

\begin{thm} \label{thm:maintheorem} The big quantum products
  \eqref{eq:QCX} for $\cX$ and \eqref{eq:QCY} for $Y$ coincide after
  analytic continuation in the $t_i$, the affine-linear change-of-variables
  \[
  t_i =
  \begin{cases}
    u_0, & i=0 \\
    -{2 \pi \sqrt{-1} \over n} + \sum_{j=1}^{n-1} L_{i j} u_j,  & i>0,
  \end{cases}
  \]
  where
  \begin{align*}
    L_{ij} = {\zeta^{2 i j} \( \zeta^{-j} - \zeta^{j}\) \over n}, &&  
    \zeta = \exp\({\pi \sqrt{-1} \over n}\),
  \end{align*}
  and the linear isomorphism 
  \begin{equation}
    \label{eq:lineariso}
    \begin{aligned}
      L:\HorblocTX & \to \HlocTY \\
      \delta_0 & \mapsto \gamma_0,\\
      \delta_j & \mapsto \sum_{i=1}^{n-1} L_{i j} \gamma_i, && \qquad 1 \leq j < n.
    \end{aligned}
  \end{equation}
  Furthermore, the isomorphism \eqref{eq:lineariso} matches the
  Poincar\'e pairing on $\HlocTY$ with the orbifold Poincar\'e pairing
  on $\HorblocTX$.
\end{thm}

Theorem~\ref{thm:maintheorem} establishes Conjecture~3.1 in
\cite{Bryan--Graber} for the case of polyhedral and binary polyhedral
groups of type $A$, and also Conjecture~1.9 in \cite{Perroni}.  The
path along which analytic continuation is taken is described after
Proposition~\ref{pro:ac} below.

To pass from the big quantum
cohomology algebras of $\cX$ and $Y$ to the small quantum cohomology
algebras, set $u_i = 0$, $e^{t_i} = q_i$, $1 \leq i < n$.

\begin{cor*}
  The small quantum products \eqref{eq:smallQCX} for $\cX$ and
  \eqref{eq:smallQCY} for $Y$ coincide after analytic continuation in
  the $q_i$, the linear isomorphism \eqref{eq:lineariso}, and the
  specialization 
  \begin{align*}
    q_i = \exp\(-{2 \pi \sqrt{-1} \over n}\), && 1 \leq i < n.
  \end{align*}
\end{cor*}

\begin{rem*}
  It would be more conventional to relate the big and small quantum
  cohomology algebras of $Y$ by the change of variables $q_i = Q_i
  e^{t_i}$, $1 \leq i < n$, where $Q_i$ is an element of a formal
  power series ring (or Novikov ring) introduced to ensure convergence
  of the product (see \emph{e.g.} \cite{Cox--Katz}*{Section~8.5.1});
  the resulting product would then depend on two families of variables
  $t_0, \ldots, t_{n-1}$ and $Q_1,\ldots,Q_{n-1}$.  We will not do
  this.  Firstly this is because there are no convergence problems
  here --- the right-hand side of \eqref{eq:QCY} defines an analytic
  function of $t_0, \ldots, t_{n-1}$ on an appropriate domain --- and
  secondly our choice makes clear how the specialization $q_i = c_i$
  of quantum parameters to roots of unity arises: it just reflects the
  affine-linear identification of flat co-ordinates
  \begin{align*}
    t_i = \log c_i + \sum_{j=1}^{n-1} L_{i j} u_j,
    && 1 \leq i < n.
  \end{align*}

\end{rem*}

\section{Mirror Symmetry}
\label{sec:mirrorsymmetry}

As discussed in the Introduction, by mirror symmetry we mean the fact
that one can compute certain genus-zero Gromov--Witten invariants of
$\cX$ and $Y$ by solving Picard--Fuchs equations.  In this Section we
make this precise.  We introduce two cohomology-valued generating
functions for genus-zero Gromov--Witten invariants, called the
\emph{$J$-functions} of $\cX$ and $Y$, and two cohomology-valued
solutions to the Picard--Fuchs equations called the
\emph{$I$-functions} of $\cX$ and $Y$.  The relationship between the
$I$-functions and the $J$-functions is given in
Proposition~\ref{pro:mirror} below.  We then describe how to extract
the quantum products \eqref{eq:QCX} and \eqref{eq:QCY} from the
Picard--Fuchs equations, and finally explain how this implies
Theorem~\ref{thm:maintheorem}.

\subsection{The $I$-Function and the $J$-Function}

The $J$-function $J_\cX(u,z)$ of $\cX$ is defined to be
\[
e^{u_0/z} \( z \delta_0 + u_1 \delta_1 + \cdots + u_{n-1} \delta_{n-1}
+ \sum_{\hbeta} \sum_{k=0}^{n-1} \correlator{\delta_k \over z -
  \psi_1}^\cX_{\hbeta} u_1^{\hbeta(1)}\cdots u_{n-1}^{\hbeta(n-1)}
\delta^k \).
\]
The sum here is over effective classes $\hbeta$, and we expand ${1 /
  (z-\psi_1)}$ as $\sum_m {\psi_1^m / z^{m+1}}$.  $J_\cX(u,z)$ is a
function of $u \in \HorblocTX$, $u = u_0 \delta_0 + \cdots + u_{n-1}
\delta_{n-1}$, which takes values in $\HorblocTX \otimes
\CC(\!(z^{-1})\!)$.  It is defined and analytic in an open subset of
$\HorblocTX$ where $|u_1|,\ldots,|u_{n-1}|$ are sufficiently small;
this follows from Proposition~\ref{pro:mirror} below.

The $J$-function $J_Y(t,z)$ of $Y$ is
\[
  e^{t_0/z} e^{\(t_1\gamma_1+ \cdots + t_{n-1} \gamma_{n-1}\)/z} \( z \gamma_0 +
  \sum_{\beta} \sum_{k=0}^{n-1} \correlator{\gamma_k \over z -
    \psi_1}^Y_{\beta} e^{d_1 t_1 +\cdots + d_{n-1} t_{n-1}}
  \gamma^k \),
\]
where the sum is over $\beta = d_1 \beta_1 + \cdots + d_{n-1}
\beta_{n-1}$ with each $d_i \geq 0$.  $J_Y(t,z)$ is a function of $t
\in \HlocTY$, $t = t_0 \gamma_0 + \cdots +t_{n-1} \gamma_{n-1}$, which
takes values in $\HlocTY \otimes \CC(\!(z^{-1})\!)$. It is defined 
and analytic in an open subset of $\HlocTY$ where $\Re (t_i) \ll 0$, 
$1 \leq i < n$; this again follows from Proposition~\ref{pro:mirror}.

Given a class $\hbeta$ in the orbifold Neron--Severi group of $\cX$,
or in other words given a sequence $\hbeta(1), \ldots, \hbeta(n-1)$ of
integers, it will be convenient to set
\begin{align*}
  \hbeta(0) = -{1 \over n} \sum_{k=1}^{n-1} (n-k) \hbeta(k), &&
  \hbeta(n) = -{1 \over n} \sum_{k=1}^{n-1} k \hbeta(k), &&
  i(\hbeta) =  n \fr{ - \hbeta(n)},
\end{align*}
where $\fr{r}$ denotes the fractional part of a rational number $r$.
The $I$-function $I_\cX(x,z)$ of $\cX$ is defined to be
\[
z e^{x_0/z}\, \sum_{\hbeta} 
{1 \over  z^{\hbeta(1) + \cdots + \hbeta(n-1)}}
\prod_{\substack{r:\hbeta(0) < r \leq 0 \\ \fr{r} = \fr{\hbeta(0)}}} 
(\lambda_1 + r z)
\prod_{\substack{s:\hbeta(n) < s \leq 0 \\ \fr{s} = \fr{\hbeta(n)}}} 
(\lambda_2 + s z)
\,
{ x_1^{\hbeta(1)} \cdots x_{n-1}^{\hbeta(n-1)}
  \over
  \hbeta(1)! \cdots \hbeta(n-1)! }
\, \delta_{i(\hbeta)} ;
\]
the sum here is over effective classes $\hbeta$.  This is a function
of $x = (x_0,\ldots,x_{n-1}) \in \cM_B$, $z \in \Cstar$, and
$\lambda_1,\lambda_2 \in \CC$ which takes values in $\HorbTX$.  Each
component of $I_\cX(x,z)$ with respect to the basis $\{\delta_i\}$ is
an analytic function of $(x,z,\lambda_1,\lambda_2)$ defined in a
domain where $|x_1|,\dots,|x_n|$ are sufficiently small and $x_0,z,
\lambda_1,\lambda_2$ are arbitrary.  By taking a Laurent expansion at
$z = \infty$ we can regard $I_\cX(x,z)$ as an analytic function of
$(x,\lambda_1,\lambda_2)$ which takes values in
$\HorblocTX\otimes\C(\!(z^{-1})\!)$.
$I_\cX(x,z)$ satisfies a system of
Picard--Fuchs equations, as follows.  Define differential operators
$\beth_i = z x_i\parfrac{}{x_i}$, $1 \leq i < n$, and
\begin{align*}
\beth_0 = \lambda_1 - {1 \over n} \sum_{k=1}^{n-1} (n-k) z x_k \parfrac{}{x_k}, && 
\beth_n = \lambda_2 - {1 \over n} \sum_{k=1}^{n-1} k z x_k\parfrac{}{x_k}.
\end{align*}
Then
\begin{subequations}
  \label{eq:PFX}
  \begin{multline}
    \label{eq:PFXa}
    \( \prod_{j:\hbeta(j)>0} \prod_{m=0}^{\hbeta(j)-1} \(\beth_j - m z\) \) I_{\cX}(x,z) = \\
    x_1^{\hbeta(1)} \cdots x_{n-1}^{\hbeta(n-1)} \(
    \prod_{j:\hbeta(j)<0} \prod_{m=0}^{-\hbeta(j)-1} \(\beth_j - m z\)
    \) I_{\cX}(x,z).
  \end{multline}
  for each orbifold Neron--Severi class $\hbeta$ such that $i(\hbeta)
  = 0$, and
  \begin{equation}
    \label{eq:PFXb}
    z \parfrac{}{x_0} I_{\cX}(x,z) = I_{\cX}(x,z).
  \end{equation}
\end{subequations}

The $I$-function of $Y$ is
\[
I_Y(y,z) = 
z \,e^{y_0/z}  y_1^{\gamma_1/z} \cdots y_{n-1}^{\gamma_{n-1}/z}  
\sum_\beta 
\prod_{j=0}^{n} \textstyle {\prod_{m \leq 0} (\omega_j + m z) \over 
\prod_{m \leq D_j(\beta)} (\omega_j + m z)}
  y_1^{d_1} \cdots y_{n-1}^{d_{n-1}},
\]
where $y_i^{\gamma_i/z} = \exp\(\gamma_i\log y_i/z\)$, the sum is
over $\beta = d_1 \beta_1 + \cdots + d_{n-1} \beta_{n-1}$ with each
$d_i \geq 0$, and
\[
D_j(\beta) = 
\begin{cases}
  d_1 & j=0 \\
  -2 d_1 + d_2 & j=1 \\
  d_{j-1} - 2 d_j + d_{j+1} & 1<j<n-1 \\
  d_{n-2} - 2 d_{n-1} & j=n-1 \\
  d_{n-1} & j=n.
\end{cases}
\] 
$I_Y(y,z)$ is a multi-valued function of $y = (y_0,\ldots,y_{n-1}) \in
\cM_B$, $z \in \Cstar$, and $\lambda_1,\lambda_2 \in \CC$ which takes
values in $\HTY$.  Each component of $I_Y(y,z)$ with respect to the
basis $\{\gamma_i\}$ is a multi-valued analytic function of
$(y,z,\lambda_1,\lambda_2)$ defined in a domain where $|y_1|,\dots,
|y_{n-1}|$ are sufficiently small, $|z| >
\max(|\lambda_1|,|\lambda_2|)$, and $y_0$ is arbitrary.  By taking a
Laurent expansion at $z = \infty$ we can regard $I_Y(y,z)$ as a
multi-valued analytic function of $(y, \lambda_1,\lambda_2)$ which
takes values in $\HlocTY\otimes \CC(\!(z^{-1})\!)$.
It also satisfies a system of Picard--Fuchs equations.  Define
differential operators
\[
\daleth_j =
\begin{cases}
  \lambda_1 + z y_1 \parfrac{}{y_1} & j=0 \\
  -2 z y_1 \parfrac{}{y_1} + z y_2 \parfrac{}{y_2} & j=1 \\
  z y_{j-1} \parfrac{}{y_{j-1}} - 2 z y_j \parfrac{}{y_j} + z y_{j+1} \parfrac{}{y_{j+1}} & 1<j<n-1 \\
  z y_{n-2} \parfrac{}{y_{n-2}} - 2 z y_{n-1} \parfrac{}{y_{n-1}} & j=n-1 \\
  \lambda_2 + z y_{n-1} \parfrac{}{y_{n-1}} & j=n.
\end{cases}
\]
Then 
\begin{subequations}
  \label{eq:PFY}
  \begin{multline}
    \label{eq:PFYa}
    \( \prod_{j:D_j(\beta) >0} \prod_{m=0}^{D_j(\beta) - 1}
    \(\daleth_j - m z\)
    \) I_Y(y,z) \\
    = q_1^{d_1} \cdots q_{n-1}^{d_{n-1}} \( \prod_{j:D_j(\beta) <0}
    \prod_{m=0}^{-D_j(\beta) - 1} \(\daleth_j - m z\) \) I_Y(y,z)
  \end{multline}
  for every $\beta = d_1 \beta_1 + \cdots + d_{n-1} \beta_{n-1}$, and
  \begin{equation}
    \label{eq:PFYb}
    z \parfrac{}{y_0} I_{Y}(y,z) = I_{Y}(y,z).
  \end{equation}
\end{subequations}
The Picard--Fuchs systems \eqref{eq:PFX} for $\cX$ and \eqref{eq:PFY}
for $Y$ coincide under the co-ordinate change \eqref{eq:xtoy}.  Thus
there is a global system of Picard--Fuchs equations --- a $\cD$-module
over all of $\cM_B$ --- which gives \eqref{eq:PFX} near the
large-radius limit point for $\cX$ and \eqref{eq:PFY} near the
large-radius limit point for $Y$.  This global nature of the
Picard--Fuchs system will play a key role in what follows.

By mirror symmetry, we mean the following.

\begin{pro}  \label{pro:mirror} \

  \begin{enumerate}
  \item $I_\cX(x,z)$ and $J_\cX(u,z)$ coincide after a change of
    variables expressing $u$ in terms of $x$.
  \item $I_Y(y,z)$ and $J_Y(t,z)$ coincide after a change of
    variables expressing $t$ in terms of $y$.
  \end{enumerate}
\end{pro}

\begin{proof}
  Part (1) is equation 23 in \cite{CCIT:lefschetz}.  To see this, set
  $t^i$ there to $x_i$, $0 \leq i<n$; $k_i$ there to $\hbeta(i)$, $1
  \leq i<n$; $\tau^i$ there to $u_i$, $0 \leq i<n$; $\lambda_1$ there
  to $\lambda_2$ here and \emph{vice versa}.  Then $I^\tw(t,z)$ there
  coincides with $I_\cX(x,z)$ here and $J^\tw(\tau,z)$ there coincides
  with $J_\cX(u,z)$ here.

  The argument that proves Theorem~0.2 in \cite{Givental:toric} also
  proves part (2) here.  Theorem~0.2 as stated only applies to compact
  semi-positive toric manifolds, but the proof applies essentially
  without change to the non-compact toric Calabi--Yau manifold $Y$.
\end{proof}

\begin{rem*}
  We learned from Bong Lian that, in unpublished work, he and
  Chien-Hao Liu have established mirror theorems for non-compact toric
  Calabi--Yau manifolds using the arguments of
  \cite{Lian--Liu--Yau:3}.  Once again, the proof for compact toric
  manifolds applies also to the non-compact toric Calabi--Yau case
  without significant change.  This gives an alternative proof of the
  second part of Proposition~\ref{pro:mirror}.
\end{rem*}

We can determine the changes of variables in
Proposition~\ref{pro:mirror} by expanding the $I$-functions and the
$J$-functions as Laurent series in $z^{-1}$.  We have
\[
J_\cX(u,z) = z + u_0 \delta_0 + u_1 \delta_1 + \cdots + u_{n-1} \delta_{n-1} +
O(z^{-1})
\]
and
\[
I_\cX(x,z) = z + f_0(x) \delta_0 + f_1(x) \delta_1 + \cdots + f_{n-1}(x) \delta_{n-1} +
O(z^{-1})
\]
where $f_0(x) = x_0$ and for $1 \leq k < n$,
\[
f_k(x) = \sum_{\substack{\text{$\hbeta$ effective}:\\ i(\hbeta) = k}} 
{\Gamma\(1-{k \over n}\) \over \Gamma\(1 + \hbeta(0)\)}
{\Gamma\({k \over n}\) \over \Gamma\(1 + \hbeta(n)\)}
{x_1^{\hbeta(1)} \cdots x_{n-1}^{\hbeta(n-1)} \over
  \hbeta(1)! \cdots \hbeta(n-1)!}.
\]
The change of variables which equates $I_\cX$ and $J_\cX$ is therefore
$u_i = f_i(x)$, $0 \leq i < n$.  As
\[
f_k(x) = x_k + \text{quadratic and higher order terms in
  $x_1,\ldots,x_{n-1}$}
\]
the functions $f_0(x),\ldots,f_{n-1}(x)$ define co-ordinates on a
neighbourhood of the large-radius limit point for $\cX$ in $\cM_B$.
We call these \emph{flat co-ordinates for $\cX$}.  Similarly, 
\[
J_Y(t,z) = z + t_0 \gamma_0 + t_1 \gamma_1 + \cdots + t_{n-1}
\gamma_{n-1} + O(z^{-1})
\]
and
\[
I_Y(y,z) = z +  g_0(y) \gamma_0 + g_1(y) \gamma_1 + \cdots + g_{n-1}(y) \gamma_{n-1} +
O(z^{-1}) 
\]
for some functions $g_0(y), \ldots, g_{n-1}(y)$ with $g_0(y) = y_0$ and for $1 \leq k < n$,
\[
g_k(y) = \log y_k + \text{single-valued analytic function of
  $y_1,\ldots,y_{n-1}$}.
\]
The change of variables which equates $I_Y$ and $J_Y$ is $t_i =
g_i(y)$, $0 \leq i < n$.  The functions $g_0(y), \ldots, g_{n-1}(y)$
define multi-valued co-ordinates on a neighbourhood of the
large-radius limit point for $Y$; these are the \emph{flat
  co-ordinates for $Y$}.  Note that the exponentiated flat
co-ordinates $\exp(g_k(y))$ are single-valued.

The $J$-functions satisfy differential equations which determine the
quantum products.

\begin{pro}
  \label{pro:QDE} \ 
  \begin{enumerate}
  \item 
    \[
    z \parfrac{}{u_i} z \parfrac{}{u_j} J_{\cX}(u,z) =
    \sum_{k=0}^{n-1} \(\delta_i \underset{\rm big}{\star}\)_j^{\phantom{j}k} 
    z \parfrac{}{u_k} J_{\cX}(u,z) 
    \]
    where $ \(\delta_i \underset{\rm big}{\star}\)_j^{\phantom{j}k} $ are the matrix
    entries of the product \eqref{eq:QCX}.
  \item 
    \[
    z \parfrac{}{t_i} z \parfrac{}{t_j} J_{Y}(t,z) =
    \sum_{k=0}^{n-1} \(\gamma_i \underset{\rm big}{\star}\)_j^{\phantom{j}k} 
    z \parfrac{}{t_k} J_{Y}(t,z) 
    \]
    where $ \(\gamma_i \underset{\rm big}{\star}\)_j^{\phantom{j}k} $ are the matrix
    entries of the product \eqref{eq:QCY}.
  \end{enumerate}
\end{pro}

\begin{proof}
  Part (2) is well-known to experts (\emph{cf.}
  \citelist{\cite{Pandharipande}*{Proposition~2}
    \cite{Cox--Katz}*{Chapter~10}}).  By the Divisor Equation, $z
  \parfrac{}{t_i} J_{Y}(t,z)$ is equal to
  \[
  z \, e^{t_0/z} e^{\(t_1\gamma_1+ \cdots + t_{n-1} \gamma_{n-1}\)/z} 
  \( \gamma_i +
  \sum_{\beta} \sum_{k=0}^{n-1} 
  \correlator{\gamma_i,{\gamma_k \over z -
      \psi_1}}^Y_{\beta} e^{d_1 t_1 +\cdots + d_{n-1} t_{n-1}} 
  \gamma^k \)
  \]
  and $ z \parfrac{}{t_i} z \parfrac{}{t_j} J_{Y}(t,z) $ is equal to
  \[
  z^2 \, e^{t_0/z} e^{\(t_1\gamma_1+ \cdots + t_{n-1}
    \gamma_{n-1}\)/z} 
  \sum_{\beta} \sum_{k=0}^{n-1} 
  \correlator{\gamma_i,\gamma_j,{\gamma_k \over z -
      \psi_1}}^Y_{\beta} e^{d_1 t_1 +\cdots + d_{n-1} t_{n-1}} 
  \gamma^k.
  \]
  This last expression is
    \begin{multline*}
      z \, e^{t_0/z} e^{\(t_1\gamma_1+ \cdots + t_{n-1} \gamma_{n-1}\)/z} 
    \sum_{\beta} \sum_{k=0}^{n-1} 
    \correlator{\gamma_i,\gamma_j,\gamma_k}^Y_{\beta} 
    e^{d_1 t_1 +\cdots + d_{n-1} t_{n-1}} 
    \gamma^k \\
    +   z \, e^{t_0/z} e^{\(t_1\gamma_1+ \cdots + t_{n-1} \gamma_{n-1}\)/z} 
    \sum_{\beta} \sum_{k=0}^{n-1} \sum_{m \geq 1}
    \correlator{\gamma_i,\gamma_j,{\gamma_k \psi_1^m \over z^m
      }}^Y_{\beta}     e^{d_1 t_1 +\cdots + d_{n-1} t_{n-1}} 
    \gamma^k.
  \end{multline*}
  Applying the Topological Recursion Relations
  \cite[Equation~6]{Pandharipande} yields
    \begin{multline*}
      z \, e^{t_0/z} e^{\(t_1\gamma_1+ \cdots + t_{n-1} \gamma_{n-1}\)/z} 
      \sum_{\beta} \sum_{k=0}^{n-1} 
      \correlator{\gamma_i,\gamma_j,\gamma_k}^Y_{\beta} 
      e^{d_1 t_1 +\cdots + d_{n-1} t_{n-1}} 
      \gamma^k \\
    + z \, e^{t_0/z} e^{\(t_1\gamma_1+ \cdots + t_{n-1} \gamma_{n-1}\)/z} 
    \sum_{\beta} \sum_{\beta'+\beta'' = \beta} 
    \sum_{k=0}^{n-1} \sum_{l=0}^{n-1} 
    \correlator{\gamma_i,\gamma_j,\gamma^l}^Y_{\beta'} \times \\
    \correlator{\gamma_l,{\gamma_k \over z - \psi_1
      }}^Y_{\beta''}
     e^{d_1 t_1 +\cdots + d_{n-1} t_{n-1}} 
    \gamma^k
  \end{multline*}
  and this is
  \[
  \( \sum_{\beta'} 
  \sum_{l=0}^{n-1} 
  \correlator{\gamma_i,\gamma_j,\gamma^l}^Y_{\beta'}  
  e^{d_1' t_1 +\cdots + d_{n-1}' t_{n-1}} \)
  z \parfrac{}{t_l} J_{Y}(t,z),
  \]
  where the sum in parentheses is over $\beta' = d_1' \beta_1 + \cdots
  + d_{n-1}' \beta_{n-1}$ with each $d_i' \geq 0$.  This proves (2).
  The proof of (1) is essentially identical, but uses the fake Divisor
  Equation \cite[Section~2.2]{Bryan--Graber} instead of the Divisor
  Equation and the Topological Recursion Relations for orbifolds
  \cite[Section~2.5.5]{Tseng} instead of the Topological Recursion
  Relations for varieties.  
\end{proof}

\subsection{From PF to QC}

Propositions \ref{pro:mirror} and \ref{pro:QDE} together show that we
can determine the quantum products \eqref{eq:QCX} and \eqref{eq:QCY}
by looking at the differential equations satisfied by $I_\cX$ and
$I_Y$ \emph{in flat co-ordinates}:
\begin{align}
  \label{eq:QDEIX}
  z \parfrac{}{u_i} z \parfrac{}{u_j} I_{\cX}(x(u),z) &=
  \sum_{k=0}^{n-1} \(\delta_i \underset{\rm big}{\star}\)_j^{\phantom{j}k} 
  z \parfrac{}{u_k} I_{\cX}(x(u),z) \\
  \label{eq:QDEIY}
  z \parfrac{}{t_i} z \parfrac{}{t_j} I_{Y}(y(t),z) &=
  \sum_{k=0}^{n-1} \(\gamma_i \underset{\rm big}{\star}\)_j^{\phantom{j}k} 
  z \parfrac{}{t_k} I_{Y}(y(t),z) 
\end{align}
A more invariant way to say this is as follows.  Let $\lambda_1$,
$\lambda_2$ be fixed complex numbers.  If we associate to a vector
field $v = \sum v_k(y) \parfrac{}{y_k}$ on $\cM_B$ the differential
operator $\sum z v_k(y) \parfrac{}{y_k}$ then the systems of
differential equations \eqref{eq:PFX}, \eqref{eq:PFY} define a
$\cD$-module on $\cM_B$.  The characteristic variety $\CharV$ of this
$\cD$-module is a subscheme of $T^\star \cM_B$, and we can read off
the quantum products from the algebra of functions $\cO_\CharV$.
Indeed, choosing flat co-ordinates on a neighbourhood $U$ of the
large-radius limit point for $\cX$ in $\cM_B$ identifies $\cO_U$ with
analytic functions in $u_0,\ldots,u_{n-1}$ and identifies the algebra
of fiberwise-polynomial functions on $T^\star U$ with
$\cO_U[\xi_0,\ldots,\xi_{n-1}]$; here $\xi_k$ is the fiberwise-linear
function on $T^\star U$ given by $\parfrac{}{u_k}$.  The ideal
defining the characteristic variety $\CharV$ is generated by elements
\[
P(u_0,\ldots,u_{n-1},\xi_0,\ldots,\xi_{n-1},0)
\]
where $P(u_0,\ldots,u_{n-1},\xi_0,\ldots,\xi_{n-1},z)$ runs over the
set of fiberwise-polynomial functions on $T^\star U$ which depend
polynomially on $z$ and satisfy
\[
P\(u_0,\ldots,u_{n-1},z\parfrac{}{u_0},\ldots,z\parfrac{}{u_{n-1}},z\)
I_\cX(u,z) = 0.
\]
Equation \eqref{eq:QDEIX} implies that
\[
\left.\cO_\CharV \right|_U = \cO_U[\xi_0,\ldots,\xi_{n-1}]/\mathfrak{I}
\]
where the ideal $\mathfrak{I}$ is generated by
\begin{align*}
\xi_i \xi_j = 
\sum_{k=0}^{n-1} \(\delta_i \underset{\rm big}{\star}\)_j^{\phantom{j}k} \xi_k && 0 \leq i,j<n.
\end{align*}
In other words, the quantum cohomology algebra \eqref{eq:QCX} of $\cX$
is the algebra of functions $\left.\cO_\CharV \right|_U$ on the
characteristic variety $\CharV$, \emph{written in flat co-ordinates on
  $U$}.  

Similarly, choosing flat co-ordinates on a neighbourhood $V$ of the
large-radius limit point for $Y$ in $\cM_B$ identifies $\cO_V$ with
analytic functions in $t_0,\ldots,t_{n-1}$, and identifies the algebra
of fiberwise-polynomial functions on $T^\star V$ with
$\cO_V[\eta_0,\ldots,\eta_{n-1}]$ where $\eta_k$ is the
fiberwise-linear function on $T^\star V$ given by $\parfrac{}{t_k}$.
Equation \eqref{eq:QDEIY} implies that
\[
\left.\cO_\CharV \right|_V = \cO_V[\eta_0,\ldots,\eta_{n-1}]/\mathfrak{J}
\]
where the ideal $\mathfrak{J}$ is generated by
\begin{align*}
\eta_i \eta_j = 
\sum_{k=0}^{n-1} \(\gamma_i \underset{\rm big}{\star}\)_j^{\phantom{j}k} \eta_k && 0 \leq i,j<n,
\end{align*}
and so the quantum cohomology algebra \eqref{eq:QCY} of $Y$ is the
algebra of functions $\left.\cO_\CharV \right|_V$ on the
characteristic variety $\CharV$, \emph{written in flat co-ordinates on
  $V$}.

\label{sec:discussion}

The characteristic variety $\CharV$ is a global analytic object ---
$\cO_\CharV$ gives an analytic sheaf of $\cO_{\cM_B}$-algebras,
defined over all of $\cM_B$ --- so to show that the quantum cohomology
algebras of $\cX$ and of $Y$ are related by analytic continuation
followed by the change-of-variables
\[
t_i =
\begin{cases}
  u_0, & i=0 \\
  -{2 \pi \sqrt{-1} \over n} + \sum_{j=1}^{n-1} L_{i j} u_j,  & i>0
\end{cases}
\]
we just need to show that the flat co-ordinates for $\cX$ and for $Y$
are related by analytic continuation followed by the
change-of-variables
\[
g_i(y) = 
\begin{cases}
  f_0(x), & i=0 \\
  -{2 \pi \sqrt{-1} \over n} + \sum_{j=1}^{n-1} L_{i j} f_j(x),  & i>0.
\end{cases}
\]

\begin{pro} \label{pro:ac}
  There exists a path from the large-radius limit point for $Y$ to the
  large-radius limit point for $\cX$ such that the analytic
  continuation of the flat co-ordinates $g_i(y)$, $1 \leq i<n$, along
  that path satisfy
  \[
  g_i(y) = - {2 \pi \sqrt{-1} \over n} + {1 \over n} \sum_{k=1}^{n-1}
  \zeta^{2 k i} \( \zeta^{-k }- \zeta^{k}\) f_k(x),
  \]
  where $\zeta = \exp\({\pi \sqrt{-1} \over n}\)$.
\end{pro}

\begin{proof}
  The flat co-ordinates $f_1(x), \ldots, f_{n-1}(x)$ and
  $g_1(y),\ldots,g_{n-1}(y)$ are independent of $\lambda_1$,
  $\lambda_2$, $x_0$, and $y_0$, so they can be extracted from the
  $z^0$ terms of $I_\cX$ and $I_Y$ after setting $\lambda_1 =
  \lambda_2 = x_0 = y_0 = 0$.  But $\left. I_\cX \right|_{\lambda_1 =
    \lambda_2 = x_0 = 0}$ and $\left. I_Y \right|_{\lambda_1 =
    \lambda_2 = y_0 = 0}$ satisfy the systems of differential
  equations \eqref{eq:PFXa}, \eqref{eq:PFYa} with $\lambda_1$ and
  $\lambda_2$ set to zero, and once $\lambda_1$ and $\lambda_2$ are
  set to zero the $z$-dependence in these differential equations
  cancels.  The flat co-ordinates $f_1(x), \ldots, f_{n-1}(x)$ and
  $g_1(y),\ldots,g_{n-1}(y)$ therefore satisfy
  \begin{multline}
    \label{eq:GKZ}
    \(
    \prod_{j:D_j(\beta) >0} \prod_{m=0}^{D_j(\beta) - 1} \(\gimel_j - m\) 
    \) f \\
    = q_1^{d_1} \cdots q_{n-1}^{d_{n-1}} \(
    \prod_{j:D_j(\beta) <0} \prod_{m=0}^{-D_j(\beta) - 1} \(\gimel_j - m\) 
    \) f
  \end{multline}
  for every $\beta = d_1 \beta_1 + \cdots + d_{n-1} \beta_{n-1}$,
  where
  \[
  \gimel_j =
  \begin{cases}
    y_1 \parfrac{}{y_1} & j=0 \\
    -2 y_1 \parfrac{}{y_1} + y_2 \parfrac{}{y_2} & j=1 \\
    y_{j-1} \parfrac{}{y_{j-1}} - 2 y_j \parfrac{}{y_j} + y_{j+1} \parfrac{}{y_{j+1}} & 1<j<n-1 \\
    y_{n-2} \parfrac{}{y_{n-2}} - 2 y_{n-1} \parfrac{}{y_{n-1}} & j=n-1 \\
    y_{n-1} \parfrac{}{y_{n-1}} & j=n.
  \end{cases}
  \]
  This is the GKZ system associated to $Y$.  It has rank $n$, and both
  $f_1(x),\ldots,f_{n-1}(x)$ plus the constant function and
  $g_1(y),\ldots,g_{n-1}(y)$ plus the constant function form bases of
  solutions.  Any analytic continuation $\tilde{g}_i(y)$ of $g_i(y)$
  to a neighbourhood of the large-radius limit point for $\cX$ still
  satisfies \eqref{eq:GKZ}, so 
  \[
  \tilde{g}_i(y) = \sum_{j=1}^{n-1} L_{ij} f_j(x) + m_i
  \]
  for some constants $L_{ij}$ and $m_i$.  Thus any analytic
  continuation of $g_i(y)$ is an affine-linear combination of the flat
  co-ordinates $f_1(x),\ldots,f_{n-1}(x)$.  It remains to choose a
  specific analytic continuation and determine the corresponding
  constants $L_{ij}$ and $m_i$.

  We proved in \cite{CCIT:lefschetz} that another basis of solutions
  to the GKZ system \eqref{eq:GKZ} is given by the constant function
  together with
  \begin{align*}
    \log \kappa_i(x) - \log \kappa_{i-1}(x), && 1 \leq i <n,
  \end{align*}
  where $\kappa_i(x)$ are roots of the polynomial
  \[
  W_\cX(\kappa) = \kappa^n + x_{n-1} \kappa^{n-1} + x_{n-2}
  \kappa^{n-2} + \cdots + x_{1} \kappa + 1.
  \]
  We number the roots such that as $x \to 0$, 
  \begin{align*}
    \kappa_i(x) \to \zeta^{2i+1}, && 0 \leq i < n.
  \end{align*}
  Each $\log \kappa_i(x)$, $0 \leq i < n$, is also a solution to
  \eqref{eq:GKZ}.  Equation 25 in \cite{CCIT:lefschetz} gives
  \begin{equation}
    \label{eq:roottof}
    \log \kappa_i(x) = {(2i+1) \pi \sqrt{-1} \over n} + {1 \over n}
    \sum_{k=1}^{n-1} \zeta^{(2i+1) k} f_k(x).
  \end{equation}

  Consider also the polynomial
  \begin{multline*}
    W_Y(\mu) = \mu^n + \mu^{n-1} + y_1 \mu^{n-2} +
    y_1^2 y_2 \mu^{n-3} + y_1^3 y_2^2 y_3 \mu^{n-4} + \cdots
    \\
    + y_1^{n-1} y_2^{n-2} \cdots y_{n-2}^2 y_{n-1}
  \end{multline*}
  and number its roots $\mu_i(y)$, $0 \leq i<n$ such that as $y\to
  0$ 
  \begin{align*}
    \mu_0(y) & \to -1 \\
    \mu_1(y) & \sim -y_1 \\
    \mu_2(y) & \sim -y_1 y_2 \\
    & \vdots \\
    \mu_{n-1}(y) & \sim -y_1 y_2 \cdots y_{n-1}.\\
  \end{align*}
  We have $W_\cX(\kappa) = 0$ if and only if $W_Y(1/(x_1 \kappa)) = 0$,
  where $x_i$ and $y_j$ are related by \eqref{eq:xtoy}, so still
  another basis of solutions to the GKZ system \eqref{eq:GKZ} is
  \begin{align*}
    \log \mu_i(y) - \log \mu_{i-1}(y)&& 1 \leq i < n
  \end{align*}
  together with the constant function.  The solution $g_i(y)$ is
  singled out by its behaviour $g_i(y) = \log y_i +
  O(y_1,\ldots,y_{n-1})$ as $y \to 0$, so
  \[
  g_i(y) = \log \mu_i(y) - \log \mu_{i-1}(y).
  \]
  
  Along any path from the large-radius limit point for $Y$ to the
  large-radius limit point for $\cX$, the root $\mu_i(y)$ of $W_Y$
  analytically continues to the root $1/(x_1
  \kappa_{\sigma(i)}(x))$ of $W_{\cX}$, for some permutation
  $\sigma$ of $\{0,1,\ldots,n-1\}$.  The group of monodromies around
  the discriminant locus of $W_\cX$ acts $n$-transitively on the set
  of roots of $W_\cX$, so we can choose a path such that $\sigma$ is
  the identity permutation.  Along this path, $\log \mu_i(y) - \log
  \mu_{i-1}(y)$ analytically continues to $\log \kappa_{i-1}(x) -
  \log \kappa_{i}(x)$, $1 \leq i < n$.  Applying equation
  \eqref{eq:roottof} yields
  \[
    g_i(y) = -{2 \pi \sqrt{-1} \over n} + {1 \over n} \sum_{k=1}^{n-1}
  \zeta^{2 k i} \( \zeta^{-k} - \zeta^{k}\) f_k(x).
  \]
\end{proof}

\begin{rem*}
  For an explicit path satisfying the conditions in
  Proposition~\ref{pro:ac}, we can concatenate two paths defined
  as follows.  The first runs from $(y_0,y_1, \ldots, y_{n-1}) =
  (0,0,\ldots,0,0)$ to $(y_0,y_1, \ldots, y_{n-1}) = (0,1,1,\ldots,1,1)$
  and is given by $y_0 = 0$ and
  \begin{align*}
    W_Y(\mu) = \Big(\mu - \(-1 - \epsilon \rho^2 - \epsilon^2 \rho^3 -
    \ldots - \epsilon^{n-1} \rho^n\)\Big) \prod_{k=1}^{n-1} \(\mu -
    \epsilon^k \rho^{k+1}\), && 0 \leq \epsilon \leq 1,
  \end{align*}
  where $\rho = \exp\({2 \pi \sqrt{-1} \over n+1}\)$.  The second runs
  from $(x_0,x_1, \ldots, x_{n-1}) = (0,1,1,\ldots,1,1)$ to $(x_0,x_1,
  \ldots, x_{n-1}) = (0,0,\ldots,0,0)$, and is given by $x_0 = 0$ and
  \begin{align*}
    W_\cX(\kappa) = \prod_{k=0}^{n-1} \( \kappa - 
    \exp \( \pi \sqrt{-1} \left[ {2k+1 \over n} \epsilon' + {2(n-k) \over
        n+1} (1 - \epsilon') \right] \) \), &&
    0 \leq \epsilon' \leq 1.
  \end{align*}
  Note that the points $(y_0,y_1, \ldots, y_{n-1}) = (0,1,1,\ldots,1,1)$
  and $(x_0,x_1, \ldots, x_{n-1}) = (0,1,1,\ldots,1,1)$ coincide.
\end{rem*}

\subsection{The Proof of Theorem \ref{thm:maintheorem}}
 
Combining Proposition~\ref{pro:ac} with the discussion at the end of
Section~\ref{sec:discussion} shows that the quantum cohomology
algebras of $\cX$ and $Y$ coincide after analytic continuation along
the path specified in Proposition~\ref{pro:ac} followed by  the
affine-linear change-of-variables 
\begin{align*}
  t_i &=
  \begin{cases}
    u_0, & i=0 \\
    -{2 \pi \sqrt{-1} \over n} + \sum_{j=1}^{n-1} L_{i j} u_j,  & i>0,
  \end{cases}
 &  L_{ij} & = {\zeta^{2 i j} \( \zeta^{-j} - \zeta^{j}\) \over n},
\end{align*}
and the linear isomorphism 
\begin{equation*}
  \begin{aligned}
    L:\HorblocTX & \to \HlocTY \\
    \delta_0 & \mapsto \gamma_0,\\
    \delta_j & \mapsto \sum_{i=1}^{n-1} L_{i j} \gamma_i, && \qquad 1 \leq j < n.
  \end{aligned}
\end{equation*}
To see that $L$ preserves the Poincar\'e pairings, first observe that
the bases
\begin{align*}
n \lambda_1 \lambda_2, \gamma_1,\gamma_2,\ldots,\gamma_{n-1} 
&&\text{and}&&
1, \omega_1,\omega_2,\ldots,\omega_{n-1} 
\end{align*}
for $\HlocTY$ are dual with respect to the Poincar\'e pairing on
$\HlocTY$.  Let $L^\dagger$ denote the adjoint to $L$ with respect to
the Poincar\'e pairing $(\cdot,\cdot)_Y$ and the orbifold Poincar\'e
pairing $(\cdot,\cdot)_\cX$.  It suffices to show that $(L^\dagger
\gamma,L^\dagger \gamma')_\cX = (\gamma,\gamma')_Y$ for all $\gamma,
\gamma' \in \HlocTY$.  For $1 \leq i<n$, we have $(L^\dagger
\omega_i,\delta_k)_\cX = (\omega_i,L \delta_k)_Y = L_{ik}$, and so 
\begin{align*}
L^\dagger \omega_i = n \sum_{k=1}^{n-1} L_{ik} \delta_{n-k}, && 1 \leq
i<n.
\end{align*}
Also $L^\dagger 1 = \delta_0$.  Straightforward calculation now gives
$(L^\dagger 1, L^\dagger 1)_\cX = (n \lambda_1 \lambda_2)^{-1}$,
$(L^\dagger 1, L^\dagger \omega_i)_\cX = 0$ for $1 \leq i<n$, and
\begin{align*}
  (L^\dagger \omega_i, L^\dagger \omega_j)_\cX =
  \begin{cases}
    0 & \text{if $|i-j|>1$} \\
    1 & \text{if $|i-j|=1$}\\
    -2 & \text{if $i=j$}
  \end{cases}
  && \text{for $1 \leq i,j<n$.}
\end{align*}
As the class $\omega_j$ is the $T$-equivariant Poincar\'e-dual to the
$j$th exceptional divisor we see that $L^\dagger$, and hence $L$, is
pairing-preserving.  This completes the proof of
Theorem~\ref{thm:maintheorem}.  \hfill $\Box$

\begin{rem*}
  A more conceptual explanation of this result is as follows.  One can
  construct a Frobenius manifold from a \emph{variation of
    semi-infinite Hodge structure} \cite{Barannikov} (henceforth \VHS)
  together with a choice of \emph{opposite subspace}\footnote{Mirror
    symmetry often associates to the quantum cohomology of some target
    space a ``mirror family'' of manifolds.  In this case one can
    think of the \VHS as an analog of the usual variation of Hodge
    structure on the mirror family, and the opposite subspace as an
    analog of the weight filtration.}.  We have argued elsewhere that
  in certain toric examples one can construct the Frobenius manifold
  which is the ``mirror partner'' to the quantum cohomology of $Y$
  from a \VHS parameterized by the $B$-model moduli space of $Y$,
  together with a distinguished opposite subspace associated to the
  large-radius limit point for $Y$ \cite{CCIT:wallcrossings1}.  (The
  Frobenius manifold mirror to the quantum cohomology of a toric
  orbifold $\cX$ birational to $Y$ is given by the \emph{same} \VHS
  but the opposite subspace corresponding to the large-radius limit
  point for $\cX$.)  One can apply this construction here to get a
  \VHS parametrized by $\cM_B$.  This \VHS has the special property
  that the opposite subspace at the large-radius limit point for $Y$
  agrees with the opposite subspace at the large-radius limit point
  for\footnote{In fact each maximal cone of the secondary fan gives
    rise to a toric partial resolution $Y'$ of $\cX$, a large-radius
    limit point for $Y'$, and an opposite subspace corresponding to
    this large-radius limit point.  All these opposite subspaces
    agree: we really do get a Frobenius structure defined over all of
    $\cM_B$.}
  $\cX$.  In general the
  difference between the opposite subspaces at different large radius
  limit points will be measured by an element of Givental's linear
  symplectic group, but in this case the corresponding group element
  maps the opposite subspaces isomorphically to each other.  This
  means that we get a Frobenius manifold \emph{over the whole
    (non-linear) space $\cM_B$}.  One can construct flat co-ordinates
  in a neighbourhood of any point of $\cM_B$, and the transition
  functions between such flat co-ordinate patches, such as
  \[
  g_i(y) = \sum_j L_{i j} f_j(x) + \log c_i,
  \]
  are necessarily affine-linear and (Poincar\'e) metric-preserving.
\end{rem*}

\begin{rem*}
  It is clear from the proof of Proposition~\ref{pro:ac} that changing
  the path along which analytic continuation is taken will result in a
  corresponding change in the statements of
  Theorem~\ref{thm:maintheorem} and its Corollary.  Hence the
  co-ordinate change in Theorem~\ref{thm:maintheorem} is not unique.
  This ambiguity can be understood as an automorphism of quantum
  cohomology.  The orbifold fundamental group
  \[
  G := \pi_1^{\text{\rm orb}}\(\cM_B\setminus\{\text{discriminant locus of
    $W_\cX$}\}\)
  \]
  acts simply-transitively on the set of homotopy types of paths from
  the large-radius limit point for $Y$ to that for $\cX$, and in
  particular acts transitively on the set of all possible co-ordinate
  changes obtained by analytic continuation (although this action is
  not effective).  This deserves further study: we just note here the
  intriguing fact that $G$ is isomorphic to
  $\widetilde{A}_{n-1}\rtimes \mu_n$, which also appears as a subgroup
  (generated by spherical twists and line bundles) of the group of
  autoequivalences of $D_Z^b(Y)$ \cite{IU, Bridgeland, IUU}.  Here
  $\widetilde{A}_{n-1}$ is the affine braid group and $D_Z^b(Y)$ is
  the bounded derived category of coherent sheaves on $Y$ supported on
  the exceptional set $Z$.
\end{rem*}

\end{document}